\documentclass[11pt]{amsart}
\usepackage[mathscr]{euscript}
\usepackage{amsthm}
\usepackage{amssymb}
\usepackage{hyperref}

\newcommand{\m}[1]{{\mbox{\uppercase {\bf {#1}}}}}
\newcommand{\sm}[1]{{\mbox{\scriptsize {\uppercase {\bf {#1}}}}}}
\newcommand{\vr}[1]{{\uppercase {\mathcal{#1}}}}

\newcommand{\clo}[1]{{\rm Clo\:\m #1}}
\newcommand{\clon}[2]{{\rm Clo_{#1}\m #2}}
\newcommand{\pol}[1]{{\rm Pol\:\m #1}}
\newcommand{\poln}[2]{{\rm Pol_{#1}\m #2}}
\newcommand{\con}{{\rm Con\:}}
\newcommand{\cn}[1]{{\con\m {#1}}}
\newcommand{\Cn}[1]{{{\bf\con}\m {#1}}}
\newcommand{\Cg}{{\rm Cg}}

\newcommand{\lb}{\langle}
\newcommand{\rb}{\rangle}
\newcommand{\mt}{\wedge}
\newcommand{\jn}{\vee}
\newcommand{\usub}{\subseteq}

\def\Dj{\mbox{\raise0.3ex\hbox{-}\kern-0.4em D}}

\newtheorem{thm}{Theorem}
\newtheorem{lm}[thm]{Lemma}
\newtheorem{prp}[thm]{Proposition}
\newtheorem{cor}[thm]{Corollary}
\theoremstyle{definition}

\newtheorem{ex}[thm]{Example}

\newtheorem{df}[thm]{Definition}

\begin{document}

\title[SMB algebras I]{SMB Algebras I: On the variety of SMB algebras}
\author[P. \Dj api\'c]{Petar \Dj api\'c}
\address{Department of Mathematics and Informatics\\ 
University of Novi Sad\\ Serbia}
\email{petarn@dmi.uns.ac.rs}
\author[P. Markovi\'c]{Petar Markovi\'c}
\address{Department of Mathematics and Informatics\\ 
University of Novi Sad\\ Serbia}
\email{pera@dmi.uns.ac.rs}
\author[R. McKenzie]{Ralph McKenzie}
\address{Department of Mathematics\\ 
Vanderbilt University\\ Nashville, TN 37240, USA}
\email[Ralph McKenzie]{ralph.n.mckenzie@vanderbilt.edu}
\author[A. Proki\'{c}]{Aleksandar Proki\'{c}}
\address{Faculty of Technical Sciences\\ 
University of Novi Sad\\ Serbia}
\email{aprokic@uns.ac.rs}
\thanks{Petar \Dj api\'{c} and Petar Marković were supported by the Ministry of Education, Science and Technological Development of the Republic of Serbia (Grant No. 451-03-68/2022-14/200125). Petar Markovi\'{c} was also supported by the  Science Fund of the Republic of Serbia grant no. 6062228. Aleksandar Proki\'{c} was supported by the Ministry of Education, Science and Technological Development of the Republic of Serbia (Grant No. 451-03-68/2022/14/200156).}

\begin{abstract}
We begin the investigation of the variety of semilattices of Mal'cev blocks, which we call SMB algebras. 
\end{abstract}

\keywords{finite axiomatizability, variety, tame congruence theory, Mal'cev condition}
\subjclass{08B05, 08A05, 08A45}

\maketitle

\section{Motivation, History and Overview}

This paper, the first in a series of at least two, is a conglomerate of old and new results. The middle two authors have written a note \cite{RP} in 2009 in an effort to identify the ``worst" case of Taylor algebras for the classification of the complexity of the Constraint Satisfaction Problem (CSP). We defined algebras which were almost the same as SMB algebras, though a slightly narrower definition, proved they occur naturally as subalgebras of reducts in Taylor algebras and solved some cases of the CSP. Then M.~Mar\'oti proved in 2010, see \cite{M-note2}, that some more cases of SMB algebras were tractable for CSP. Finally, in 2019, A. Bulatov resolved the CSP for SMB algebras in the general case, and also defined SMB algebras in the most general form in \cite{B1}. Since Bulatov's definition of SMB algebras is also the most natural, we use his definition in this paper (as the generalization doesn't affect the correctness of our old proofs). Bulatov went on to resolve the general CSP Dichotomy using the case of SMB algebras as a template.

In an effort to unify the proofs of the CSP Dichotomy by Zhuk and by Bulatov, Barto et al. investigated minimal Taylor algebras in \cite{DreamTeam} (see also \cite{DreamTeamconf} for the conference version). These minimal Taylor algebras have a binary absorbing subuniverse which behaves somewhat similarly to the least $\sim$-class in SMB algebras, which is the underlying reason why, with a lot of effort, Bulatov's proof of the tractability of the CSP transfers from SMB algebras to general Taylor algebras.

We decided to revisit the SMB algebras, both to finally publish the early results, since it seems that SMB algebras are more significant than we first thought, and to see if there are more results to be proved about them which may be generalized to Taylor algebras. Our effort is divided between (1) trying to unify the two CSP Dichotomy proofs into a simpler argument which can be generalized to more complex topics in model checking and (2) trying to see what other problems we can solve for SMB algebras and then generalize to Taylor algebras. In the course of our research we amassed a lot of results and decided to split them into two papers by topic. This paper contains algebraic results which establish some useful properties of SMB algebras. These properties will be used in the second paper, which is devoted to the complexity of the CSP. The final part of this paper contains some partial results toward proving Park's Conjecture for SMB algebras.

After this, introductory section, we give the necessary mathematical background and notation in Section 2. Section 3 is giving some justification why the SMB algebras we consider are not so unusual. It seems probable that the results of this section can also be obtained using ideas from \cite{KSz} and \cite{KK}, but here we pursue a different approach, using the weak near-unanimity operation, thus making our arguments more elementary and self-contained. The results of Section 3 first appeared in the unpublished manuscript by the middle two authors of this paper, see \cite{RP}.

Section 4 proves that the class of all SMB algebras is a finitely based variety with a Taylor term. Section 5 defines regular SMB algebras and proves that they form a finitely based variety by finding a finite base. Moreover, in Section 5 we prove that in any finitely generated variety $\vr v$ of SMB algebras there exist two terms such that the $\vr v$-algebras with these two terms replacing the fundamental operations form a variety of regular SMB algebras.

Section 6 investigates the principal congruences in SMB algebras and Section 7 proves partial results towards Park's conjecture for SMB algebras.

\section{Preliminaries}

\subsection{Universal Algebra}

We assume that the reader is familiar with the basics of Universal Algebra. The readers who need these facts and definitions are referred to classic textbooks \cite{burris-sank} and \cite{alvin}. Moreover, we need some basic results of Tame Congruence Theory, an advanced theory in Universal Algebra which was developed in \cite{HM}.

Following \cite{alvin}, we use the notation $\clo a$ and $\clon{n}{a}$ for the clone of term operations and the set of $n$-ary term operations of the algebra $\m a$, respectively. Also, $\pol a$ and $\poln{n}{a}$ denote the clone of polynomial operations and the set of $n$-ary polynomial operations of the algebra $\m a$, respectively, while $\Cn a$ denotes the congruence lattice of $\m a$.

An operation $f$ of an algebra $\m a$ is said to be idempotent if the identity $f(x,x,\ldots,x)\approx x$ holds in $\m a$. An algebra is idempotent if all of its fundamental operations (equivalently, term operations) are idempotent. We will assume all algebras are idempotent in this paper. Secondly, unless we explicitly state otherwise, all algebras in this paper will be assumed to be finite by default.

The third restriction on the algebras under consideration is that we are interested in algebras which generate varieties that omit type $\mathbf{1}$ covers (in the language of Tame Congruence Theory), equivalently algebras which have a Taylor term. A Taylor term for $\m a$ is a term $t$ which is idempotent in $\m a$ and satisfies in $\m a$ a finite set $\Sigma$ of identities of the form
\[t(u_1,\dots,u_n)\approx t(v_1,\dots,v_n)\]
such that each $u_i,v_j\in \{x,y\}$ and such that for each $i\leq n$ there exists an identity in $\Sigma$ in which $u_i=x$ and $v_i=y$. According to \cite{MM}, we may assume for any algebra $\m a$ that there exists a particular kind of a Taylor term $w$ such that the identities
$$w(x,x,\ldots,x,y)\approx w(x,x,\ldots,x,y,x)\approx \ldots \approx w(y,x,x,\ldots,x)$$
hold in $\m a$. Those identities, plus idempotence (which we need not assume again) make $w$ a {\em weak near-unanimity} term of $\m a$, or a wnu term for short. We introduce the notation $x\circ_w y$ for the binary term operation $w(x,x,\ldots,x,y)$. A wnu term $w$ of $\m a$ is {\em special} if $\m a$ satisfies the additional identity $x\circ_w(x\circ_w y)\approx x\circ_w y$. We repeat here the following well-known result (probably folklore) since its proof will be useful to us.

\begin{lm}[Lemma 4.7 of \cite{MM}]\label{fulltreeiteration}
Any finite algebra which has a wnu term must also have a special wnu term.
\end{lm}

\begin{proof}
When $f$ is an $n$-ary and $g$ an $m$-ary operation, then $f\lhd g$ stands for the $mn$-ary operation given by 
$$
\begin{gathered}
(f\lhd g)(x_1,\dots,x_{mn})=\\
f(g(x_1,\dots,x_m),g(x_{m+1},\dots,x_{2m}),\dots,g(x_{mn-m+1},\dots,x_{mn})).\end{gathered}
$$
Let $w$ be the wnu term of $\m a$. Consider the term 
\[v=w\lhd w\lhd\dots\lhd w,\]
where the composition $\lhd$ occurs $|A|!-1$ times. Note that $v$ has $k^{|A|!}$ many variables if $w$ was $k$-ary. From idempotence and the fact that $w$ is wnu it follows that $v$ is also wnu and that 
\[x\circ_v y=x\circ_w(x\circ_w(\dots x\circ_w(x\circ_w y)\dots)),\]
where the operation $\circ_w$ occurs $|A|!-1$ times in the above formula. It is a basic fact about selfmaps of a finite set that, for any map $f:A\rightarrow A$, the composition of $|A|!$ copies of $f$ is idempotent for composition, i.e. 
\[f^{|A|!}(f^{|A|!}(y))=f^{|A|!}(y).\]
Applying this to the maps of the form $f(y)=a\circ_w y$ we obtain
\[a\circ_v(a\circ_v y)=f^{2|A|!}(y)=f^{|A|!}(y)=a\circ_v y.\]
Since $a\in A$ can be chosen arbitrarily, the conclusion follows.
\end{proof}

\subsection{Tame Congruence Theory} Tame Congruence Theory analyzes finite algebras according to local behavior of the polynomial operations of the algebra. We assume the reader is familiar with Tame Congruence Theory, as exposed in the book \cite{HM} and refer to this book for further information the reader may require. Polynomial operations are term operations in which some of the variables may have been substituted by fixed constants (elements of the algebra). The tame congruence theory classifies the covers in the congruence lattice of a finite algebra by first finding a minimal image of an idempotent unary polynomial which distinguishes the two congruences (this is the minimal set). The structure of the minimal set together with the polynomial operations of the algebra which are compatible with that minimal set depend only on the congruences which constitute the cover. It may be of five types, unary (type {\bf 1}), affine (type {\bf 2}), Boolean (type {\bf 3}), lattice (type {\bf 4}) and semilattice (type {\bf 5}). Absence of the unary type, not only in the finite algebra $\m a$, but also in any finite algebra in the variety $\m a$ generates, is equivalent to the existence of the wnu term in $\m a$, see \cite{MM}. Thus we may assume no type {\bf 1} occurs.

If $\alpha\prec\beta$ in $\Cn a$, and the type of that cover is {\bf 5}, then in any minimal set $U\in M_{\sm a}(\alpha,\beta)$, one $\beta$-class restricts to $U$ as $B$, which intersects exactly two $\alpha$-classes, while $\alpha$ and $\beta$ restrict the same way to $U\setminus B$. We call $B$ the body of $U$. The polynomials of $\m a$ restricted to $B/(\alpha|_B)$ are those of a two-element semilattice. If there is a body of some $(\alpha,\beta)$-minimal set which intersects two $\alpha$-classes $C_1$ and $C_2$, then any other $(\alpha,\beta)$-minimal set which intersects the classes  $C_1$ and $C_2$ imposes the same semilattice order on the two of them. Hence, we can say that an $\alpha$-class $C_1$ is below another $\alpha$-class $C_2$ (both within the same $\beta$-class) whenever there exists an $(\alpha,\beta)$ minimal set $U$ and its body $B$ so that the in the semilattice order $C_1\cap B$ is below $C_2\cap B$. Moreover, for any $\beta$-class, the transitive closure of the semilattice orders coming from all bodies of minimal sets is a connected partial order $\leq$ of all $\alpha$-classes inside that $\beta$-class, and $\leq$ is compatible with all operations of $\m a/\alpha$ (i.e. $\leq$ is a subuniverse of $(\m a/\alpha)^2$).

\section{Maximal type 5 covers}

This section serves as motivation to show that SMB algebras are both natural and ubiquitous in Taylor algebras. At the time when we defined them, we were looking for a ``worst case" of Taylor algebras for proving tractability of the Constraint Satisfaction Problem. To avoid the tractability results of \cite{BK}, \cite{IMMVW} and \cite{M-note1}, we were looking for idempotent finite Taylor algebras in which a type $\mathbf{5}$ cover $\beta\prec_5 \gamma$ is directly above a type $\mathbf{2}$ cover $\alpha\prec_2 \beta$ in the congruence lattice (this is almost, but not quite, forced in some finite algebra in the finitely generated variety $\vr v$ when $\vr v$ avoids the cases covered by those three papers). By factoring out $\alpha$ and restricting to one class of $\gamma$, and going to a proper subalgebra, one constructs an idempotent finite Taylor algebra $\m{a}$ with a maximal congruence $\sim$ such that $\m{A}/{\sim}$ is a type $\mathbf{5}$ simple algebra and each ${\sim}$-class is an Abelian algebra, or more generally a Mal'cev algebra (recall that each congruence block is a subuniverse in idempotent algebras). In this section we show that only a small additional assumption is needed to force $\m a$ to be an SMB algebra, which we will define.

We will assume in this section that $\m a=\lb A;w\rb$ is a finite algebra with a wnu operation $w$ as its only fundamental operation. We also assume that ${\sim}\in\cn a$ is an Abelian congruence such that ${\sim}\prec 1_{\sm a}$ and that $\mathrm{typ}(\sim,1_{\sm a})={\mathbf 5}$. Since the top congruence of the cover is $1_{\sm a}$, then, in all minimal sets, the body is the same as the whole minimal set.

From now on and throughout this section, the partial order $\leq$, strict order $<$ and covering relation $\prec$ will always refer to the partial order on $\sim$-classes defined at the end of the last section. We say $x\leq y$ ($x<y$, $x\prec y$) in $A$ iff $[x]_{\sim}\leq [y]_{\sim}$ ($[x]_{\sim}<[y]_{\sim}$, $[x]_{\sim}\prec [y]_{\sim}$). Throughout the paper, we will write $a\sim b$ instead of $(a,b)\in{\sim}$.

\begin{lm}\label{cover}
Let $O\prec I$ be two $\sim$-classes. Then $O\cup I$ is a subuniverse of $\m a$ and $w(x_1,\ldots,x_n)\in I$ iff $\{x_1,\ldots,x_n\}\usub I$, in other words on the semilattice $\lb\{O,I\};\mt\rb$, $w^{\sm a/{\sim}}|_{\{O,I\}}(x_1,\ldots,x_n)=x_1\mt\ldots\mt x_n$.
\end{lm}

\begin{proof}
Since $O\prec I$ and our algebra is idempotent, any term applied to elements of $O\cup I$ must be in a ${\sim}$-class which lies in the interval $[O,I]$. But since $O\prec I$, this interval consists only of classes $O$ and $I$, so $O\cup I$ is a subuniverse of $\m a$. 

Moreover, there must exist a subtrace $\{0,1\}\usub O\cup I$ such that $0\in O$, $1\in I$ and the pseudo-meet polynomial of the minimal set $U\in M_{\sm a}({\sim},1_{\sm a})$ containing $\{0,1\}$ acts like the meet with respect to order $0<1$ (we don't have to mention that they are in the trace - or body - of $U$, as the upper congruence is $1_{\sm a}$). We know that there exists an idempotent polynomial $e\in\poln{1}{a}$ such that $e(A)=U$. Moreover, for the polynomial $ew(x_1,\ldots,x_n)\in\poln{n}{a}|_U$, we know that $ew$ is equal to meet of some of its variables, or is constant. Since $ew(0,0,\ldots,0)=0$, $ew(1,1,\ldots,1)=1$ and $0\nsim 1$, we know that $ew$ is not constant modulo ${\sim}$. We see that $ew$ must be the meet of some of its variables modulo ${\sim}|_{U}$, and let $x_j$ be included in this meet. On the other hand, assume that $x_i$ is not included in the meet, so $ew$ does not depend on $x_i$ on $U/({\sim}|_{U})$. Then $1=ew(1,1,\ldots,1)\sim ew(1,1,\ldots,1,0,1,1,\ldots,1)$, where $0$ is in the $i$th place. But we know that $ew$ does depend on $x_j$, so $ew(1,1,\ldots,1,0,1,1,\ldots,1)\sim 0$, where $0$ is now in the $j$th place. The weak near-unanimity of the operation $w$ implies that $0\sim 1$, a contradiction. So $ew(x_1,\ldots,x_n)=x_1\mt\ldots\mt x_n$ on $U/(\sim_U)$.

Our task will be complete when we prove the same for $w|_{O\cup I}$. Let $\{x_1,\ldots,x_n\}\usub O\cup I$. Clearly, if $\{x_1,\ldots,x_n\}\usub I$, then $w(x_1,\ldots,x_n)\in I$. Otherwise, for some $i\leq n$, $x_i\in O$. Therefore, $(0,0,\ldots,0)\leq (x_1,\ldots,x_n)\leq (y_1,\ldots,y_n)$, where $y_i=0$ and $y_j=1$ for all $j\neq i$. Therefore, $0=w(0,0,\ldots,0)\leq w(x_1,\ldots,x_n)\leq w(y_1,\ldots,y_n)=1\circ_w 0\sim 0$. So we get that $w(x_1,\ldots,x_n)\in O$ in this case.
\end{proof}

\begin{cor}\label{strict}
Let $y<x$ be two elements of $A$. Then $y\leq x\circ_w y<x$.
\end{cor}

\begin{proof}
This follows from the last part of the proof of Lemma~\ref{cover} and monotonicity of operations. Namely, there exists $z\in A$ such that $y\leq z\prec x$. Therefore, $(x\circ_w y)\leq (x\circ_w z)\sim z\prec x$.
\end{proof}

We define weak near-unanimity operations $w_i$ and binary operations $\circ_i$ in $\clo A$ so that $x\circ_i y=w_i(y,x,x,\ldots,x)$, $w_1(x_1,\ldots,x_n)=w(x_1,\ldots,x_n)$ and $$\begin{gathered}w_{i+1}(x_1,\ldots,x_n)=\\
w_i(w_i(x_1,\ldots,x_n)\circ_ix_1,w_i(x_1,\ldots,x_n)\circ_ix_2,\ldots,w_i(x_1,\ldots,x_n)\circ_ix_n).\end{gathered}$$

\begin{thm}\label{iterate}
If $O,I\in A/{\sim}$ and $O<I$, then $O\cup I$ is a subuniverse of $\lb A;w_{|A|}\rb$, and where $\{x_1,\ldots,x_n\}\usub \{O,I\}$ and $\mt$ is the meet semilattice operation on $\{O,I\}$ defined by the order $O<I$, we have $w_{|A|}^{\sm A/{\sim}}(x_1,\ldots,x_n)=x_1\mt\ldots\mt x_n$.
\end{thm}

\begin{proof}
Define $y\prec_k x$ to mean that $y<x$ and any covering chain from $y$ to $x$ has length at most $k$ (in particular, $\prec_1=\prec\cup{\sim}$). We claim that if $y\prec_k x$, then $x\circ_k y\sim y$.

We prove this claim by induction on $k$. The base case is a consequence of Lemma~\ref{cover}. Let the claim hold for $k$. Let $y\prec_{k+1}x$. Then $y\prec_1 z\prec_k x$ for some $z$, and 
$$\begin{gathered}
y\leq x\circ_{k+1} y=((x\circ_k y)\circ_k x)\circ_k((x\circ_k y)\circ_k y)\leq\\ ((x\circ_k z)\circ_k x)\circ_k((x\circ_k z)\circ_k y)\sim
(z\circ_k x)\circ_k(z\circ_k y)\sim \\(z\circ_k x)\circ_k y\leq (x\circ_k z)\circ_k y\sim z\circ_k y\sim y
\end{gathered}$$
The inequality $z\circ_k x\leq x\circ_k z$ follows from $z\leq x$ and monotonicity of $w_k$,  as $z\circ_k x=w_k(z,z,z,\ldots,z,x)\leq w_k(z,x,x,\ldots,x,x)=x\circ_k z$. The remaining inequalities and equalities modulo ${\sim}$ are from the inductive assumption and monotonicity.

Now for any $x\in I$ and $y\in O$, we know that $y<x$, so $y\prec_{|A|}x$. Therefore, $x\circ_{|A|}y\sim y$, so $x\circ_{|A|}y\in O$. Let $(x_1,\ldots,x_n)\in (O\cup I)^n$ be arbitrary and $(y_1,\ldots,y_n)\in\{x,y\}^n$ be such that $x_i\sim y_i$ for all $i$. If there exists $i$ such that $x_i\in O$, then 
$$
\begin{gathered}
y\leq w_{|A|}(x_1,\ldots,x_n)\sim w_{|A|}(y_1,\ldots,y_n)\leq\\
w_{|A|}(x,x,\ldots,x,y,x,x,\ldots, x)=x\circ_{|A|}y\sim y,
\end{gathered}
$$
where the only occurrence of $y$ in the final $w_{|A|}$ is at the $i$th position. Therefore, $w_{|A|}(x_1,\ldots,x_n)\in O$. In the remaining case, all $x_i$ are in $I$, and since $\m a$ is an idempotent algebra and $w_{|A|}$ is its term, then $w_{|A|}(x_1,\ldots,x_n)\in I$.
\end{proof}

\begin{cor}\label{chain}
Let $B_0<B_1<\ldots<B_k$ be ${\sim}$-classes. Then $B=B_1\cup\ldots\cup B_k$ is a subuniverse of $\lb A;w_{|A|}\rb$.
\end{cor}

\begin{proof}
Let $x_1,\ldots,x_n\in B$ and assume that $x_i$ is in the least, while $x_j$ in the greatest ${\sim}$-class among them. Then $x_i\leq w_{|A|}(x_1,\ldots,x_n)\leq w_{|A|}(x_j,x_j,\ldots,$ $x_j,x_i,x_j,x_j,\ldots,x_j)=x_j\circ_{|A|}x_i\sim x_i$.
\end{proof}

\begin{cor}\label{special}
Assume that $w$ is a special weak near-unanimity term. If $x>y$ then $B=x/{\sim}\cup (x\circ_w y)/{\sim}$ is a subuniverse of $\m a$ and $w$ acts on $B$ (modulo ${\sim}$) like the meet of all variables.
\end{cor}

\begin{proof}
Using the previous ideas, it suffices to prove that, where $x_1,\ldots,x_n\in B$ and exactly one of $x_i$ is in $(x\circ_w y)/{\sim}$, then $w(x_1,\ldots,x_n)\in (x\circ_w y)/{\sim}$. But, in this case \[w(x_1,\ldots,x_n)\sim w(x,x,\ldots,x,x\circ_w y,x,x,\ldots,x)=x\circ_w(x\circ_w y)=x\circ_w y.\]
\end{proof}

For the next three Lemmas and Theorems, $\circ$ will denote $\circ_{w_{|A|}}$.

\begin{lm}\label{up-down}
Let $B$,$C$ and $D$ be ${\sim}$-classes and $B\leq D$, $C\leq D$. Then for any $x\in B$ and $y\in C$, $x\circ y$ is in a ${\sim}$-class $E$ which is the greatest lower bound of $B$ and $C$.
\end{lm}

\begin{proof}
Let $z\in D$ be arbitrary. According to Theorem~\ref{iterate}, $x\leq x\circ z\leq z\circ x\sim x$, so $x\circ z\sim x$, and also $z\circ y\sim y$. Therefore, from $x\circ y\leq x\circ z\sim x$ and $x\circ y\leq z\circ y\sim y$, we get that $x\circ y$ is in a ${\sim}$-class which is a lower bound for $[x]_{\sim}$ and $[y]_{\sim}$. On the other hand, let us assume that $a\leq x$ and $a\leq y$ is any element of a class which is a lower bound for $[x]_{\sim}$ and $[y]_{\sim}$. Then $a=a\circ a\leq x\circ y$, so $[x\circ y]_{\sim}$ is the greatest lower bound for $[x]_{\sim}$ and $[y]_{\sim}$.
\end{proof}

\begin{thm}\label{smallest}
In the order $\leq$ there always exists the least ${\sim}$-class.
\end{thm}

\begin{proof}
Assume that $B$ is a minimal ${\sim}$-class and that it is not the smallest. Then there exists a ${\sim}$-class $C$ which is not greater than or equal to $B$. By connectedness of $\leq$ on $A/{\sim}$, we know that there is a fence $B=B_0\leq B_1\geq B_2\leq B_3\geq\ldots\geq B_{2k}=C$. Let us consider such a fence with $k$ minimal. According to Lemma~\ref{up-down} there exists a common lower bound $B_1'$ for $B$ and $B_2$. But since $B_1'$ is a lower bound for $B$, and $B$ is minimal, then $B_1'=B$ and $B\leq B_2\leq B_3$. So, $B=B_0\leq B_3\geq\ldots\geq B_{2k}=C$ is also a fence which contradicts minimality of $k$.
\end{proof}

\begin{thm}\label{greatest}
Let $\leq$ have the largest element. Then every pair of ${\sim}$-classes has a greatest lower bound and $w_{|A|}(x_1,\ldots,x_n)\in [x_1]_{\sim}\mt\dots\mt [x_n]_{\sim}$. In particular, if $\m a$ is simple (i. e. ${\sim}=0_{\sm a}$), and $\leq$ has a greatest element, then $\m a$ has a semilattice term-reduct.
\end{thm}

\begin{proof}
This follows in a straightforward way from Lemma~\ref{up-down}.
\end{proof}

\begin{thm}\label{SMB-justify}
Let $\m a=(A,w)$ be a wnu algebra, ${\sim}\in\cn a$ a maximal congruence, $\mathrm{typ}({\sim}, 1_{\sm a})=\mathbf{5}$ and $\sim$ is an Abelian congruence. If there exists a largest $\sim$-class with respect to $\leq$, then there is a binary term operation $\mt$ of $\m a$ such that $(\m a/\sim,\mt)$ is a semilattice, and on each $\sim$-class the restriction of $\mt$ acts as the second projection.
\end{thm}

\begin{proof}
First we iterate $w$ like in the proof of Theorem~\ref{iterate} until we obtain the wnu term $w_{|A|}$. Then Theorem~\ref{greatest} guarantees that $(\m a/\sim,\circ_{w_{|A|}})$ is a semilattice. Next we iterate $w_{|A|}$ like in Lemma~\ref{fulltreeiteration} until we obtain a special wnu term $v$. Since the iteration used in the proof of Lemma~\ref{fulltreeiteration} implies that 
\[x\circ_v y = x\circ_{w_{|A|}}(x\circ_{w_{|A|}}(\dots x\circ_{w_{|A|}}(x\circ_{w_{|A|}}y)\dots)),\]
hence the operations $\circ_w$ and $\circ_v$ act the same way in $\m a/{\sim}$, i.e. $(\m a/\sim,\circ_v)$ is a semilattice.

Now we assume some $\sim$-class $B\in A/{\sim}$ is fixed and consider the operation $x\circ_v y$, where $x,y\in B$. By Abelianness, there exists a ring $\m r$ and a faithful (the ring elements injectively correspond to the endomorphisms) $\m r$-module $\m b$ with universe $B$ such that all term operations of $\m a$ which map $B^n$ into $B$ are polynomial operations of the module $\m b$.

Hence, the restriction of $\circ_v$ to $B$ equals in $\m b$ to
\[x\circ_v y=\alpha x+\beta y+c,\]
Plugging the zero of the module for $x$ and $y$ we obtain $c=0$ from idempotence, while the idempotence of $\m a$ implies $(\alpha+\beta)x=x$ and the faithfulness implies that 
\[\alpha+\beta=1.\]
From the special condition $x\circ_v(x\circ_v y)=x\circ_v y$ and we obtain $(\alpha+\alpha\beta)x+\beta^2 y = \alpha x+\beta y$, so plugging zero for $x$ and by faithfulness we obtain
\[\beta^2=\beta.\]
From this follows $\alpha\beta=(1-\beta)\beta=\beta-\beta^2=0$ and $\alpha^2=1-2\beta+\beta^2=1-\beta=\alpha$. Now we define
\[x\mt y:=(y\circ_v x)\circ_v y.\]
On each $\sim$-class $B$ we have by the above argument that
\[x\mt y = \alpha^2 y+\alpha\beta x + \beta y=(\alpha+\beta) y=y.\]
On the other hand, in $\m a/{\sim}$ by the laws of semilattices $\mt$ acts in precisely the same way as $\circ_v$ and $\circ_{w_{|A|!}}$.
\end{proof}

On the other hand, in any finite Taylor algebra, there is a ternary term operation $d(x,y,z)$ which acts as a Mal'cev operation on each congruence class of an Abelian (or even a solvable) congruence. This follows from \cite{HM}, Theorem 9.6 and Lemma 9.4. This observation and Theorem~\ref{SMB-justify} inspire the definition of semilattices of Mal'cev blocks, or SMB algebras:

\renewcommand{\labelenumi}{($\roman{enumi}$)}
\begin{df}\label{SMBdef}
Let $\m A=(A; \mt ,d)$ be an idempotent algebra and ${\sim}\in\cn a$. We say that $\m a$ is a semilattice of Mal'cev blocks with respect to ${\sim}$, SMB algebra over ${\sim}$ for short, if
\begin{enumerate}
\item $(A/{\sim};\mt^{\sm a/{\sim}})$ is a semilattice and
\item on each ${\sim}$-class $D$, the operation $\mt{\upharpoonright_D}$ acts as the second projection, while $d{\upharpoonright_D}$ acts as a Mal'cev operation.
\end{enumerate}
We say that $\m A=(A; \mt ,d)$ is a semilattice of Mal'cev blocks, SMB algebra for short, if $\m A$ is an idempotent algebra such that there exists a congruence ${\sim} \in \cn A$ so that $\m a$ is an SMB algebra over ${\sim}$. We denote the class of all SMB algebras by $\vr S$.
\end{df}

We feel that the assumptions we needed for Theorem~\ref{SMB-justify} are not too restrictive. If type $\mathbf{2}$ does not appear in $\vr v(\m a)$, then $\vr v(\m a)$ is congruence meet-semidistributive, a significantly better understood class than Taylor varieties. Similarly, if type $\mathbf{5}$ does not appear in $\vr v(\m a)$, then $\vr v(\m a)$ satisfies a nontrivial congruence identity, another more restrictive class. Whenever both types $\mathbf{2}$ and $\mathbf{5}$ appear in the congruence lattice of the same Taylor algebra, and some type $\mathbf{2}$ cover $\alpha\prec\beta$ is directly below a type $\mathbf{5}$ cover $\beta\prec\gamma$, by restricting to the subuniverse which is one $\gamma$-class and factoring out $\alpha$, we arrive at a situation which is not exactly the same as the one we had in Theorem~\ref{SMB-justify}, but similar enough to make our proofs work. The idea is to note that $\beta$ is still Abelian, and one can take a nonminimal $\beta$-class $C$ (with respect to the type $\mathbf{5}$ order) and the union of all $\beta$-classes below or equal to $C$ within the chosen $\gamma$-class constitute a subalgebra on which Theorem~\ref{SMB-justify} applies. Thus we obtain an SMB algebra consisting of $C$ and all $\beta$-classes below it.

Now we turn to a bad example which shows that, when the type $\mathbf{5}$ order is flat, pretty much anything can happen, i.e. we have almost no control of the wnu operation  or its derived operation $\circ_w$:

\begin{prp}
Let $\m a=\lb A;w\rb$ be any finite wnu algebra. Then there exists a finite simple algebra $\m b=\lb B;v\rb$ of type {\bf 5} such that $v$ is a weak near-unanimity operation on $B$ of the same arity as $w$ and $\m a\leq \m b$.
\end{prp}

\begin{proof}
Let $A=\{a_1,\ldots,a_n\}$. We define $B=A\cup \{0,s,a_{n+1}\}$, and define a weak near-unanimity term $v$ on $B$ so that $v|_{A}=w$ and
\begin{itemize}
\item $v(x,x,\ldots,x)=x$.
\item $0$ is an absorbing element for $v$, that is, whenever there is a $0$ among the entries of $v$, the result is $0$.
\item For all evaluations of $v$ which are neither fully in $A$, nor nearly unanimous, the result is $0$.
\item For all $1\leq i\leq n$, $s\circ_v a_i=a_i\circ_v s=a_{i+1}$.
\item For all $1\leq i\leq n$, $a_{n+1}\circ_v a_i=a_{n+1}$ and $a_i\circ_{v} a_{n+1}=s$.
\item $s\circ_v a_{n+1}=0$ and $a_{n+1}\circ_v s=a_1$.
\end{itemize}

We will prove that $\m b$ is simple. We make the following two claims: $(1)$ any nontrivial congruence $\alpha$ on $\m b$ contains $(a_{n+1},0)$ and $(2)$ $\Cg^{\sm b}(a_{n+1},0)=1_{\sm b}$. Define the unary polynomials $l_1(x):=s\circ_v x$, $l_2(x):=a_{n+1}\circ_v x$, $r_1(x):=x\circ_v s$ and $r_2(x):=x\circ_v a_{n+1}$. 

To prove the first claim, assume that $(x,y)\in\alpha\in\cn b$ for some $x\neq y\in B$. If $(x,y)=(a_i,a_j)$ for some $i<j$, or $(a_i,0)$, then $(l_1^{n+1-i}(x),l_1^{n+1-i}(y))=(a_{n+1},0)$. If $(x,y)=(s,a_i)$, for some $1\leq i\leq n+1$,  or $(x,y)=(s,0)$, then $(l_1^{n}l_2(x),l_1^{n}l_2(y))=(a_{n+1},0)$.

To prove the second claim, let $\alpha=\Cg^{\sm b}(a_{n+1},0)$. Then we have $(a_j,0)=(l_1^{j-1}r_1(a_{n+1}),l_1^{j-1}r_1(0))\in\alpha$ for each $1\leq j\leq n$. Moreover, from $(a_1,0)\in \alpha$, we get $(s,0)=(r_2(a_1),r_2(0))\in\alpha$.

Finally, the subset $U=\{0,a_{n+1}\}$ is clearly equal to $l_2^2(B)$ for the idempotent unary polynomial $l_2^2(x)$. Therefore, it must be a $(0_{\sm a},1_{\sm a})$-minimal set. It is also a subuniverse of $\m b$ with $v$ playing the role of $n$-ary meet of variables with respect to order $0<a_{n+1}$. So the type of $(0_{\sm a},1_{\sm a})$ must be {\bf 3}, {\bf 4}, or {\bf 5}. However, the possibilities that it may be other than {\bf 5} are excluded since $0$ is absorbing for any polynomial $p$ of $\m b|_U$ in each variable on which $p$ actually depends, so there can be no polynomial $q\in\poln{2}{b}|_U$ which is a binary join.
\end{proof}

\section{Semilattices of Mal'cev blocks are a variety}

\begin{prp}\label{SMBquasi}
The class $\vr s$ of all SMB algebras is a quasivariety.
\end{prp}

\begin{proof}
Let $\m a$ be an SMB algebra over ${\sim}\in\cn a$ and let $a,b\in A$. If $a\sim b$, then $a\mt b=b$ and $b\mt a=a$, using Definition~\ref{SMBdef} $(ii)$. On the other hand, if $a\mt b=b$ and $b\mt a=a$, then 
\[
[a]_{\sim}=[b\mt a]_{\sim}=[b]_{\sim}\mt [a]_{\sim}=[a]_{\sim}\mt [b]_{\sim}=[a\mt b]_{\sim}=[b]_{\sim}.
\]
So, 
\begin{equation}\label{charalpha}
a\sim b\text{ iff }a\mt b=b\text{ and }b\mt a=a.
\end{equation}
Then it is easy to express that ${\sim}$ is a congruence by two quasiidentities per operation:
\[
\begin{gathered}
(x_1\mt x_2=x_2\;\&\; x_2\mt x_1=x_1\;\&\;y_1\mt y_2=y_2\;\&\; y_2\mt y_1=y_1) \Rightarrow\\ 
(x_1\mt y_1)\mt (x_2\mt y_2)=(x_2\mt y_2)\\
(x_1\mt x_2=x_2\;\&\; x_2\mt x_1=x_1\;\&\;y_1\mt y_2=y_2\;\&\;y_2\mt y_1=y_1) \Rightarrow \\
(x_2\mt y_2)\mt (x_1\mt y_1)=(x_1\mt y_1)
\end{gathered}
\]
and similarly, two quasiidentities express that ${\sim}$ is compatible with $d$.

The fact that $\m a/{\sim}$ is a meet semilattice can now be expressed as a finite set of identities: we are expressing the facts that $(x\mt x)\sim x$, $(x\mt y)\sim (y\mt x)$ and $((x\mt y)\mt z)\sim (x\mt (y\mt z))$ using two identities for each relation as in \eqref{charalpha}. Thus Definition~\ref{SMBdef} $(i)$ is expressed with (quasi)identities. Finally, for Definition~\ref{SMBdef} $(ii)$, we  need to express ``if $x\sim y$, then $x\mt y=y$'', ``if $x\sim y$ and $x\sim z$, then $x\sim d(x,y,z)$'', ``if $x\sim y$, then $d(y,x,x)=y$'' and ``if $x\sim y$, then $d(x,x,y)=y$''. All of these are clearly equivalent to quasiidentities using the characterization \eqref{charalpha} for $x\sim y$.
\end{proof}

\begin{lm}\label{leminjection}
Let $\m A=\langle A;\mt,d\rangle$ be an SMB algebra over ${\sim}\in \cn a$ and let $\theta\in \cn a$. If $(a,b)\in{\sim}\jn\theta$, then there exists some $e\in[a]_{{\sim}\jn\theta}$ such that $[e]_{\sim}\leq[a\mt b]_{\sim}$ and
\[\{[x]_\theta:x\in [a]_{\sim}\}\cup\{[x]_\theta:x\in [b]_{\sim}\}\subseteq\{[x]_\theta:x\in [e]_{\sim}\}.\]
\end{lm}

\begin{proof}
Let 
\[
a=c_0,d_0,c_1,d_1,\dots,c_n,d_n=b
\]
be such that 
\[
\text{for all }0\leq i\leq n, [c_i]_{\sim}=[d_i]_{\sim}
\] 
and 
\[
\text{for all }0\leq i<n, [d_i]_\theta=[c_{i+1}]_\theta.
\]
We think of $n$ as the ``distance'' between $a$ and $b$. We will prove, by an induction on $n$, the following 

{\bf Claim.} Let 
\[e:=(\dots((c_0\mt c_1)\mt c_2)\mt\dots )\mt c_n.\]
Then
\[\{[x]_\theta:x\in [a]_{\sim}\}\subseteq\{[x]_\theta:x\in [e]_{\sim}\}.\]

{\em Proof of the Claim}. We first prove the Claim when $n=1$. Let $e=c_0\mt c_1$. Clearly, 
\begin{equation}\label{eq1}
[e]_\theta=[c_0\mt c_1]_\theta=[c_0\mt d_0]_\theta=[d_0]_\theta.
\end{equation} 
For any $x\in [a]_{\sim}=[c_0]_{\sim}$, we have
\[
[e\mt x]_{\sim}=[(c_0\mt c_1)\mt x]_{\sim}=[c_0\mt c_1]_{\sim}\mt [c_0]_{\sim}=[c_0\mt c_1]_{\sim}=[e]_{\sim}
\]
and
\[
[e\mt x]_\theta\stackrel{\eqref{eq1}}{=}[d_0\mt x]_\theta=[x]_\theta.
\]
So, any $\theta$-class which intersects $[a]_{\sim}$ in $x$ also intersects $[e]_{\sim}$ in $e\mt x$ and the base case is proved.

Now suppose that the Claim is true for all $a$ and $b$ which are at the distance $n-1$ in the same ${\sim}\jn\theta$-class, and prove it for distance $n$. Define 
\[
\begin{gathered}
c_i':=d_i\mt c_{i+1}\text{ for all }1\leq i<n,\\
a'=c_0':=c_1\mt a\\
d_i':=c_i'\mt d_{i+1}\text{ for all }0\leq i<n,\\
b':=c_{n-1}'\mt b\:(=d_{n-1}',\text{ since }b=d_n).
\end{gathered}
\]
We see that 
\[
\begin{gathered}
\mbox{}[d_0']_{\sim}=[(c_1\mt a)\mt d_1]_{\sim}=[(c_1\mt a)\mt c_1]_{\sim}=[c_1\mt a]_{\sim}=[a']_{\sim}, \\
[d_i']_{\sim}=[(d_i\mt c_{i+1})\mt d_{i+1}]_{\sim}=[(d_i\mt c_{i+1})\mt c_{i+1}]_{\sim} = [d_i\mt c_{i+1}]_{\sim}=[c_i']_{\sim},\\
[d_i']_\theta=[(d_i\mt c_{i+1})\mt d_{i+1}]_\theta=[c_{i+1}\mt d_{i+1}]_\theta = [d_{i+1}]_\theta=[c_{i+2}]_\theta,\\
[c_{i+1}']_\theta=[d_{i+1}\mt c_{i+2}]_\theta=[c_{i+2}]_\theta = [d_i']_\theta.
\end{gathered}
\]
Also note that all $c_i'$ and all $d_j'$ are in the same ${\sim}\jn\theta$-class as $a$ and $b$ (since all $c_i$ and $d_j$ are in that class, which is closed under $\mt$ by idempotence). We conclude that the inductive assumption can be applied to $a'$ and $b'$. Thus 
\[\{[x]_\theta:x\in [a']_{\sim}\}\subseteq\{[x]_\theta:x\in [e']_{\sim}\},\]
where $e'=(\dots((c_0'\mt c_1')\mt c_2')\mt\dots )\mt c_{n-1}'$. As 
\[
\begin{gathered}
\mbox{}[e']_{\sim}=[c_0']_{\sim}\mt\dots\mt[c_{n-1}']_{\sim}=\\
[c_1]_{\sim}\mt[a]_{\sim}\mt[c_1]_{\sim}\mt[c_2]_{\sim}\mt[c_2]_{\sim}\mt[c_3]_{\sim}\mt\dots\mt[c_{n-1}]_{\sim}\mt[c_n]_{\sim}=\\
[c_0]_{\sim}\mt[c_1]_{\sim}\mt\dots\mt[c_n]_{\sim}=[e]_{\sim},
\end{gathered}
\]
we obtain that 
\begin{equation}\label{eq2}
\{[x]_\theta:x\in [a']_{\sim}\}\subseteq\{[x]_\theta:x\in [e]_{\sim}\}.
\end{equation}
Now note that $[a']_{\sim}=[a\mt c_1]_{\sim}=[c_0\mt c_1]_{\sim}$. From the proof of the base case of the Claim we obtain
\begin{equation}\label{eq3}
\{[x]_\theta:x\in [a]_{\sim}\}\subseteq\{[x]_\theta:x\in [a']_{\sim}\}.
\end{equation}
Putting \eqref{eq2} and \eqref{eq3} together, we get
\[\{[x]_\theta:x\in [a]_{\sim}\}\subseteq\{[x]_\theta:x\in [e]_{\sim}\},\]
completing the inductive proof of the Claim. 

Now notice that the statement of the Claim is symmetric: by flipping $a$ and $b$ and also looking at the chain of $c_i$s and $d_i$s ``backwards'', we would get 
\[\{[x]_\theta:x\in [b]_{\sim}\}\subseteq\{[x]_\theta:x\in [e']_{\sim}\},\]
where 
\[
e'=(\dots((d_n\mt d_{n-1})\mt d_{n-2})\mt\dots )\mt d_0.
\]
However, since $(A/{\sim};\mt)$ is a semilattice and $(c_i,d_i)\in{\sim}$ for all $0\leq i\leq n$, we get $[e]_{\sim}=[e']_{\sim}$. Therefore, by proving the Claim for $[a]_{\sim}$ and some distance $n$, we have also proved the Claim for $[b]_{\sim}$ and the same distance $n$, using the same $[e]_{\sim}$, and thus Lemma~\ref{leminjection} follows from the Claim.
\end{proof}

\begin{thm}\label{SMBvar}
The class $\vr s$ of all SMB algebras is a variety.
\end{thm}

\begin{proof}
According to Proposition~\ref{SMBquasi}, $\vr s$ is a quasivariety and hence closed under the class operators $\mathsf S$ and $\mathsf{P}$. It remains to prove that $\mathsf{H}(\vr s)\subseteq \vr s$.

Let $\m a$ be an SMB algebra over ${\sim}\in \cn a$. Let $\m b\in \mathsf{H}(\m a)$. We know that $\m b\cong \m a/\theta$, where $\theta\in\cn a$ is the kernel of the surjective homomorphism from $\m a$ to $\m b$. It suffices to prove that $\m a/\theta$ is an SMB algebra, as the notion is invariant under isomorphism.

From Isomorphism Theorems, we know that $({\sim}\jn\theta)/\theta\in\cn a/\theta$. We prove that $\m a/\theta$ is an SMB algebra over $({\sim}\jn\theta)/\theta$. We have that 
\[
\begin{gathered}((A;\mt)/\theta)/(({\sim}\jn\theta)/\theta)\cong(A;\mt)/({\sim}\jn\theta)\cong\\
((A;\mt)/{\sim})/(({\sim}\jn\theta)/{\sim})\in\mathsf{H}((A;\mt)/{\sim}).
\end{gathered}
\]
Therefore, $((A;\mt)/\theta)/(({\sim}\jn\theta)/\theta)$ is a semilattice.

Now assume that $([a]_\theta,[b]_\theta)\in({\sim}\jn\theta)/\theta$. Hence $(a,b)\in {\sim}\jn\theta$. According to Lemma~\ref{leminjection}, there exists $e\in[a]_{{\sim}\jn\theta}$ and $a',b'\in [e]_{\sim}$ such that $[a]_\theta=[a']_\theta$ and $[b']_\theta=[b]_\theta$. As $\m a$ is an SMB algebra over ${\sim}$, we know that $a'\mt b'=b'$ and $d(b',b',a')=a'=d(a',b',b')$. Using $(a,a')\in\theta$ and $(b,b')\in\theta$, we obtain
\[
[a]_\theta\mt [b]_\theta=[b]_\theta\text{ and }d([b]_\theta,[b]_\theta,[a]_\theta)=[a]_\theta=d([a]_\theta,[b]_\theta,[b]_\theta).
\]
Since $[a]_\theta$ and $[b]_\theta$ were chosen arbitrarily in the same $({\sim}\jn\theta)/\theta$-class, this means that $\mt$ acts as the second projection on each $({\sim}\jn\theta)/\theta$-class, while $d$ acts as a Mal'cev operation on each $({\sim}\jn\theta)/\theta$-class. (Of course, each $({\sim}\jn\theta)/\theta$-class is closed under $d$ because of idempotence.) We proved that $\m a/\theta$ is an SMB algebra over $({\sim}\jn\theta)/\theta$, as desired.
\end{proof}

\begin{cor}\label{SMBFB}
The variety $\vr s$ of SMB algebras is finitely based.
\end{cor}

\begin{proof}
By Proposition~\ref{SMBquasi} and Theorem~\ref{SMBvar}, $\vr s$ is a variety which is finitely axiomatizable by quasiidentities. Then by Compactness Theorem $\vr s$ is finitely based. 
\end{proof}

Note that we don't actually know a finite equational basis for $\vr s$, we just know that one exists. In the next section we will define a subvariety of $\vr s$ for which we know a finite equational base.

\begin{prp}\label{SMBTaylor}
The variety $\vr s$ of SMB algebras has a Taylor term.
\end{prp}

\begin{proof}
We define the term $t$ in the following way:
\[t(x_1, x_2,x_3,x_4,x_5,x_6):=d(x_1\mt x_2, x_3\mt x_4, x_5\mt x_6).\]
Note that $[x\mt y]_{\sim}=[y\mt x]_{\sim}$. From Definition~\ref{SMBdef} $(ii)$ and idempotence we obtain
\begin{align*}
t(x,y,x,y,x,y)&= d(x\mt y, x\mt y, x\mt y)\approx x\mt y\\
t(y,x,y,x,x,y)&= d(y\mt x, y\mt x, x\mt y) \approx x\mt y\\
t(x,y,y,x,y,x)&= d(x\mt y, y\mt x, y\mt x)\approx x\mt y.
\end{align*}
Hence, $t$ is a Taylor term in $\vr s$.
\end{proof}

\section{Regular SMB algebras}

\begin{df}\label{regSMBdef}
We say that an SMB algebra $\m a=\lb A;\mt,d\rb$ over ${\sim}\in\cn a$ is {\em regular} if 
\begin{enumerate}
\item for all $a,b,c\in A$, $[d(a,b,c)]_{\sim}=[(a\mt b)\mt c]_{\sim}$,
\item for all $a,b\in A$ such that $[a]_{\sim}\geq[b]_{\sim}$, $a\mt b=b$,
\item $\m a\models d(x,y,z)\approx d((y\mt z)\mt x,(x\mt z)\mt y,(x\mt y)\mt z)$, and
\item $\m a\models(x\mt y)\mt y\approx x\mt y$.
\end{enumerate}
\end{df}

\begin{lm}\label{regSMBmeetbyd}
Any regular SMB algebra satisfies also the following identity:
\[\m a\models x\mt y\approx d(x,x,y)\approx d(y,x,x).\]
Therefore, in regular SMB algebras, the clone of all terms is generated by $d$.
\end{lm}

\begin{proof}
Using Definition~\ref{regSMBdef} $(iii)$ and $(iv)$, idempotence and the fact that $d$ is Mal'cev on $\sim$-classes, we compute
\[
\begin{gathered}
d(x,x,y)\approx d((x\mt y)\mt x,(x\mt y)\mt x,(x\mt x)\mt y)\approx\\
d((x\mt y)\mt x,(x\mt y)\mt x, x\mt y)\approx x\mt y\text{ and}\\
d(y,x,x)\approx d((x\mt x)\mt y,(y\mt x)\mt x,(y\mt x)\mt x)\approx\\
d(x\mt y,y\mt x,y\mt x)\approx x\mt y.
\end{gathered}
\]
\end{proof}

\begin{prp}\label{SMBtoreg}
Let $\m a$ be an SMB algebra, not necessarily finite. Then there are terms $d'(x,y,z)$ and $x\mt'y$ of $\m a$ so that the algebra $\m a'=\lb A;\mt',d'\rb$ is an SMB algebra which satisfies identities $(i)-(iii)$ in the definition of a regular SMB algebra. Moreover, the congruence ${\sim}$ remains unchanged and whenever $a,b,c$ are in the same ${\sim}$-class, then $d'(a,b,c)=d(a,b,c)$. 
\end{prp}

\begin{proof}
\item We define $\mt'$ in the following way:
\[
x\mt' y :=(x\mt y)\mt y.
\]
Now, whenever $[a]_{\sim} \geq [b]_{\sim}$, we have $[a\mt b]_{\sim}=[b]_{\sim}$, implying that
\begin{equation}\label{eq4}
a\mt' b= (a\mt b)\mt b=b.
\end{equation}
Hence, Definition~\ref{regSMBdef} $(ii)$ holds. Also, for all $a,b\in A$, we have 
\[[a\mt' b]_{\sim}=[a\mt b]_{\sim}\mt[b]_{\sim}=[a\mt b]_{\sim}.\]
Hence, modulo ${\sim}$, $\mt$ and $\mt'$ are the same semilattice operation, and from \eqref{eq4} it follows that $\mt'$ acts as the second projection on each ${\sim}$-class.

Next, we define $d'$ in the following way:
\[d'(x,y,z):=d((y\mt' z)\mt' x, (x\mt' z)\mt' y, (x\mt' y)\mt' z).\]
If $a$, $b$ and $c$ are in the same ${\sim}$-class, we have \[
\begin{gathered}
(b\mt' c)\mt' a=a,\\
(a\mt' c)\mt' b=b,\text{ and}\\
(a\mt' b)\mt' c=c,
\end{gathered}
\]
and hence $d'(a,b,c)=d(a,b,c)$ whenever $a$, $b$ and $c$ are in the same $\sim$-class. Moreover, this implies that $d'$ is a Mal'cev operation on ${\sim}$-classes. Together with the above observations about $\mt'$, we have that $\lb A;\mt',d'\rb$ is an SMB algebra which satisfies $(iii)$ of Definition~\ref{regSMBdef}. 

Finally, for all $a,b,c\in A$,
\[[(b\mt' c)\mt' a]_{\sim}=[(a\mt' c)\mt' b]_{\sim}=[(a\mt' b)\mt' c)]_{\sim}=[a]_{\sim}\mt[b]_{\sim}\mt[c]_{\sim},\]
so by the idempotence of $d$, $[d'(a,b,c)]_{\sim}=[(a\mt b)\mt c]_{\sim}$. On the other hand, we proved $[x\mt' y]_{\sim}= [x\mt y]_{\sim}$ for all $x,y\in A$, and therefore
\[
[(a\mt' b)\mt' c]_{\sim} = [(a\mt b)\mt c]_{\sim}=[a]_{\sim}\mt[b]_{\sim}\mt[c]_{\sim}.
\]
Thus, Definition~\ref{regSMBdef} $(i)$ holds, too.
\end{proof}

\begin{prp}\label{SMBtospec}
Let $\m a$ be a finite SMB algebra. Then there are $\m a$-terms $d'(x,y,z)$ and $x\mt'y$ such that $\lb A;\mt',d'\rb$ is a regular SMB algebra. Moreover, the congruence ${\sim}$ remains unchanged and whenever $a,b,c$ are in the same ${\sim}$-class, then $d'(a,b,c)=d(a,b,c)$.
\end{prp}

\begin{proof}
Using Proposition~\ref{SMBtoreg}, we assume that $\m a$ already satisfies $(i)-(iii)$ from the definition of a regular SMB algebra. We define $\mt'$ in the following way:
\[
x\mt'y:=((\dots (x\mt y)\mt y\dots )\mt y)\mt y,
\]
where the operation $\mt$ occurs $|A|!$ times. We use the well-known fact that by composing $|A|!$ many times a map $f:A\rightarrow A$ with itself, we obtain a map $g:A\rightarrow A$ which satisfies $g(g(x))=g(x)$ for all $x\in 
A$. Therefore,
\[(x\mt' y)\mt' y\approx x\mt' y,\] 
while still for all $a,b\in A$, $[a\mt'b]_{\sim}=[a\mt b]_{\sim}$ and if $[a]_{\sim} \geq [b]_{\sim}$, then $a\mt' b=b$. Hence, $\mt'$ satisfies the properties $(ii)$ and $(iv)$ of Definition~\ref{regSMBdef}.

Next, $d'$ is defined from $d$ and $\mt'$ the same way as in the proof of Proposition~\ref{SMBtoreg} and the proof of the properties $(i)$ and $(iii)$ of Definition~\ref{regSMBdef} proceeds the same way as in Proposition~\ref{SMBtoreg}.
\end{proof}  

\begin{lm}\label{regtermalpha}
Let $\m a$ be a regular SMB algebra over ${\sim}\in\cn a$. Then for every $\m a$-term $t(x_1,\dots,x_n)$ and all $a_1,\dots,a_n\in A$,
\[[t(a_1,\dots,a_n)]_{\sim}=[a_1]_{\sim}\mt\dots\mt[a_n]_{\sim}.\]
\end{lm}

\begin{proof}
It is a straightforward induction on the complexity of $t$, since all variables and both fundamental operations of $\m a$ satisfy the statement.
\end{proof}

\begin{prp}\label{RegSMBbase}
The following list of identities is an equational base for the variety $\vr r$ of regular SMB algebras:

\begin{description}
\item[Idem1)] $x\mt x\approx x$
\item[Idem2)] $d(x,x,x)\approx x$

\item[Comm)] $(x\mt y)\mt (y\mt x)\approx y\mt x$

\item[Assoc1)] $(x\mt (y\mt z))\mt ((x\mt y)\mt z)\approx(x\mt y)\mt z$
\item[Assoc2)] $((x\mt y)\mt z)\mt (x\mt (y\mt z))\approx x\mt (y\mt z)$

\item[Mal)] $d(x\mt y, y\mt x, y\mt x)\approx d(y\mt x, y\mt x, x\mt y)\approx x\mt y$

\item[Regi1)] $((x\mt y)\mt z)\mt d(x,y,z)\approx d(x,y,z)$
\item[Regi2)] $d(x,y,z)\mt ((x\mt y)\mt z)\approx (x\mt y)\mt z$

\item[Regii1)] $x\mt (x\mt y)\approx x\mt y$
\item[Regii2)] $x\mt (y\mt x)\approx y\mt x$

\item[Regiii)] $d(x,y,z)\approx d((y\mt z)\mt x, (x\mt z)\mt y, (x\mt y)\mt z)$

\item[Regiv)] $(x\mt y)\mt y\approx x\mt y$
\end{description}
\end{prp}

\begin{proof}
Let $\m a$ be any model of the above equations. First we define the relation $\sim$: $x\sim y$ iff $x\mt y= y$ and $y\mt x = x$. We need to show that $\sim$ is a congruence. By definition and {\bf Idem1)}, it is clearly a reflexive and symmetric relation.

{\bf Claim.} If $x\mt y=y$ and $y\mt z = z$, then $x\mt z = z$.

From the assumptions, {\bf Assoc1)} and {\bf Regiv)}, we get
\[\begin{gathered}
z=y\mt z=(x\mt y)\mt z=(x\mt(y\mt z))\mt((x\mt y)\mt z)=\\
(x\mt z)\mt (y\mt z)=(x\mt z)\mt z=x\mt z,
\end{gathered}
\]
and the Claim is proved. From two applications of the Claim we obtain that $\sim$ is transitive, so it is an equivalence relation. 

Next we need to prove that $\sim$ is compatible with the operations. Assume that $a\sim a'$, $b\sim b'$ and $c\sim c'$. Then we have $a\mt a'=a'$ by assumption and $a'\mt (a'\mt b')=a'\mt b'$ by {\bf Regii1)}. By Claim, it follows that
\begin{equation}\label{eq6}
a\mt (a'\mt b')=a'\mt b'
\end{equation}
Similarly, from the assumption we know that $b\mt b'=b'$ and from {\bf Regii2)} that $b'\mt(a'\mt b')=a'\mt b'$. By Claim we obtain
\begin{equation}\label{eq7}
b\mt(a'\mt b')=a'\mt b'.
\end{equation}
From {\bf Assoc2)} we get
\begin{equation}\label{eq8}
((a\mt b)\mt (a'\mt b'))\mt(a\mt (b\mt (a'\mt b')))=a\mt (b\mt (a'\mt b')).
\end{equation} 
Using first {\bf Regiv)}, and then the above equations, we compute
\[
\begin{gathered}
(a\mt b)\mt (a'\mt b')=((a\mt b)\mt (a'\mt b'))\mt (a'\mt b')\stackrel{\eqref{eq6}}{=}\\
((a\mt b)\mt (a'\mt b'))\mt (a\mt(a'\mt b'))\stackrel{\eqref{eq7}}{=}\\
((a\mt b)\mt (a'\mt b'))\mt (a\mt (b\mt(a'\mt b')))\stackrel{\eqref{eq8}}{=}\\
a\mt (b\mt(a'\mt b'))\stackrel{\eqref{eq7}}{=}
a\mt(a'\mt b')\stackrel{\eqref{eq6}}{=}a'\mt b'.
\end{gathered}
\]
Analogously (by flipping primes in the above argument) we get
\[(a'\mt b')\mt (a\mt b)=a\mt b,\]
so $(a\mt b)\sim (a'\mt b')$, hence $\sim$ is compatible with $\mt$.

From {\bf Regi1)} and {\bf Regi2)} we know that $d(a,b,c)\sim ((a\mt b)\mt c)$ and $d(a',b',c')\sim ((a'\mt b')\mt c')$, so the compatibility of $\sim$ with $\mt$ implies that $\sim$ is also compatible with $d$. Thus, $\sim$ is a congruence of $\m a$.

Now the identities {\bf Idem1)}, {\bf Comm)}, {\bf Assoc1)} and {\bf Assoc2)} imply that $\m a/\sim$ is a semilattice. From the definition of $\sim$ follows that $\mt$ is the second projection on each $\sim$-class. Taken together with that fact, {\bf Mal)} implies that $d$ is Mal'cev on each $\sim$-class. Therefore, $\m a$ is an SMB algebra.

As we noted a couple of paragraphs above, {\bf Regi1)} and {\bf Regi2)} are equivalent to Definition~\ref{regSMBdef}, property $(i)$ (given our definition of the relation $\sim$). 

Next, assume that $[a]_{\sim}\geq [b]_{\sim}$. Therefore, $[a\mt b]_{\sim}=[b]_{\sim}$, and by definition of ${\sim}$, $(a\mt b)\mt b=b$. By {\bf Regiv)} we know that $(a\mt b)\mt b=a\mt b$, so by transitivity we obtain $a\mt b=b$, thus proving Definition~\ref{regSMBdef}, property $(ii)$.

Definition~\ref{regSMBdef}, properties $(iii)$ and $(iv)$ are actually identities which are included in the equational base as {\bf Regiii)} and {\bf Regiv)}. Therefore, $\m a$ is a regular SMB algebra. On the other hand, the base identities are easily verifiable in each regular SMB algebra, so the proposition is proved.
\end{proof}

\begin{cor}\label{regSMBtermsimclass}
Let $\m a$ be a regular SMB algebra and $t(x_1,\dots,x_n)$ a term in which each variable actually appears. Then for all $a_1,\dots,a_n\in A$,
\[
[t(a_1,\dots,a_n)]_{\sim}=[a_1]_{\sim}\mt\dots\mt[a_n]_{\sim}.\]
\end{cor}

\begin{proof}
The proof is a standard induction on the complexity of $t$. When $t$ is a variable it is clear, while the step follows from the definition of SMB algebras and from the property $(i)$ of Definition~\ref{regSMBdef} (depending on which operation, $\mt$ or $d$, we use in the step).
\end{proof}

\section{Principal congruences of regular SMB algebras}

This section is devoted to limiting the length of Mal'cev chains needed to generate a principal congruence of an SMB algebra. We start with a well-known lemma, whose proof we include for reader's convenience.

\begin{lm}\label{MalcevCG}
Let $\m m=(M;d)$ be a Mal'cev algebra. $R\subseteq M\times M$ is a congruence of $\m m$ iff $R$ is a subuniverse of $\m m\times\m m$ which contains the diagonal.
\end{lm}

\begin{proof}
Assume $R$ is reflexive and compatible with $d$. We need to prove that $R$ is symmetric and transitive. If $(a,b)\in R$ then by reflexivity and compatibility \[(b,a)=(d(a,a,b),d(a,b,b))=d^{\sm m\times\sm m}((a,a),(a,b),(b,b))\in R.\]
If $(a,b),(b,c)\in R$, then by reflexivity and compatibility
\[(a,c)=(d(a,b,b),d(b,b,c))=d^{\sm m\times\sm m}((a,b),(b,b),(b,c))\in R.\]
\end{proof}

\begin{df}\label{SMBDabddef}
Let $\m a$ be an SMB algebra and $a,b\in A$. By $D_{\{a,b\}}$ we denote the subuniverse of $\m a\times\m a$ generated by $\{(a,b),(b,a)\}\cup\{(c,c):c\in A\}$.
\end{df}

We note that $D_{\{a,b\}}=\{(p(a,b),p(b,a)):p\in \poln{2}{a}\}$. Moreover, if $\m a$ is a regular SMB algebra, $(c,d)\in D_{\{a,b\}}$ and $c\neq d$, then $((a\mt b)\mt c)\sim((a\mt b)\mt d)$. Let us only sketch out this argument, as it is easy. Assume that $(c,d)=(p(a,b),p(b,a))$, $p(x,y)=t(x,y,c_1,\dots,c_k)$ for some term $t$, and $c_i\in A$ are constants. Then we can apply Corollary~\ref{regSMBtermsimclass} to the term $t(x,y,z_1,\dots,z_k)$ and make a case analysis based on how many among $x$ and $y$ actually occur in $t$. In all three cases we obtain from Corollary~\ref{regSMBtermsimclass} that $c\sim d$.

\begin{thm}\label{regSMBCg}
Let $\m a$ be a regular SMB algebra and $a,b\in A$. Then \[\Cg(a,b)=D_{\{a,b\}}\circ D_{\{a,b\}}\circ D_{\{a,b\}}.\]
\end{thm}

\begin{proof}
$D_{\{a,b\}}\circ D_{\{a,b\}}\circ D_{\{a,b\}}$ is clearly reflexive and symmetric because $D_{\{a,b\}}$ is, and any finite composition of compatible relations is compatible. It remains to prove transitivity. Suppose that \[(x_1,x_2),(x_2,x_3),(x_3,x_4),(x_4,x_5),(x_5,x_6),(x_6,x_7)\in D_{\{a,b\}}.\]
Moreover, assume that at most two of these pairs are on the diagonal (or we would be done), and that $x_1\neq x_2$ and $x_6\neq x_7$ (or we would move the diagonal pair into the middle of our sequence $x_1,\dots,x_7$).

In view of the remarks immediately preceding the theorem and Definition~\ref{regSMBdef} $(ii)$, there are essentially two cases:

{\bf Case 1}: $a\mt x_1=x_1$, $a\mt x_7=x_7$, $b\mt x_2=x_2$ and $b\mt x_6=x_6$. From Definition~\ref{regSMBdef} $(ii)$ we know that $b\mt (a\mt x_2)=a\mt x_2$ and, similarly, $b\mt (a\mt x_6)=a\mt x_6$. Then we obtain the following pairs in $D_{\{a,b\}}$:
\begin{equation}\label{eq9}
\begin{gathered}
(x_1,a\mt x_2),(a\mt x_2,b\mt (a\mt x_3)),(b\mt (a\mt x_3),b\mt (a\mt x_4)),\\(b\mt (a\mt x_4),b\mt (a\mt x_5)),(b\mt (a\mt x_5),a\mt x_6),(a\mt x_6,x_7)\in D_{\{a,b\}}.
\end{gathered}
\end{equation}
Since for $i=2,3,4,5$ there exist polynomials $p_i\in\poln{2}{a}$ such that \[\{p_i(a,b),p_i(b,a)\}=\{x_i,x_{i+1}\},\] by Corollary~\ref{regSMBtermsimclass} we obtain that $b\mt (a\mt x_i)\sim b\mt (a\mt x_{i+1})$ for $i=2,3,4,5$. Therefore, all pairs in \eqref{eq9}, except for $(x_1,a\mt x_2)$ and $(a\mt x_6,x_7)$ are in the same $\sim$-class $B$. By Lemma~\ref{MalcevCG}, $D_{\{a,b\}}\cap(B\times B)$ is a congruence of the Mal'cev algebra $(B;d)$ and therefore it is transitive. Hence, $(a\mt x_2,a\mt x_6)\in D_{\{a,b\}}$, and since we already know that $(x_1,a\mt x_2),(a\mt x_6,x_7)\in D_{\{a,b\}}$, we are done in this case.

{\bf Case 2}: $a\mt x_1=x_1$, $b\mt x_7=x_7$, $b\mt x_2=x_2$ and $a\mt x_6=x_6$. Similarly as above we have that
\begin{equation}\label{eq10}
\begin{gathered}
(x_1,a\mt x_2),(a\mt x_2,b\mt (a\mt x_3)),(b\mt (a\mt x_3),b\mt (a\mt x_4)),\\(b\mt (a\mt x_4),b\mt (a\mt x_5)),(b\mt (a\mt x_5),b\mt x_6),(b\mt x_6,x_7)\in D_{\{a,b\}}.
\end{gathered}
\end{equation}
The only difference perhaps requiring comment is that we used $b\mt x_6=b\mt (a\mt x_6)$ to conclude $(b\mt (a\mt x_5),b\mt x_6)\in D_{\{a,b\}}$. Now, as above, $a\mt x_2$, $b\mt (a\mt x_3)$, $b\mt (a\mt x_4)$, $b\mt (a\mt x_5)$ and $b\mt x_6$ all lie in the same $\sim$-class and the case concludes by an application of Lemma~\ref{MalcevCG}, just like the previous one.
\end{proof}

\begin{cor}\label{6polyCg}
If $\m a$ is a regular SMB algebra and $a,b,c,d$ are such that $(c,d)\in\Cg(a,b)$, then there exist unary polynomials $p_1,\dots,p_6\in\poln{1}{a}$ and elements $c=e_0,e_1,\dots,e_5,e_6=d\in A$ so that $\{p_i(a),p_i(b)\}=\{e_{i-1},e_i\}$ for $i=1,2,\dots,6$.
\end{cor}

\begin{proof}
We use Theorem~\ref{regSMBCg} and the description of $D_{\{a,b\}}$ from the remarks above Theorem~\ref{regSMBCg} to obtain $e_2,e_4\in A$ and binary polynomials $q_1,q_2,q_3\in \poln{2}{a}$ such that $\{q_i(a,b),q_i(b,a)\}=\{e_{2i-2},e_{2i}\}$ for $i=1,2,3$ (where $e_0=c$ and $e_6=d$). Then we introduce $e_{2i-1}:=q_i(a,a)$ for $i=1,2,3$. The result follows.
\end{proof}

\section{Toward Park's Conjecture for SMB algebras}

In this section we write our partial results towards Park's Conjecture for SMB algebras. If we manage to complete its proof, we will probably write another paper, and call it SMB III.

The next lemma is very basic and can be stated much more generally, we include it just to clarify some steps in later proofs.

\begin{lm}\label{SMBCgsim}
Let $\m a$ be an SMB algebra over $\sim$ (not necessarily regular) and $a,b,c,d\in A$. Then the following are equivalent:
\begin{enumerate}
\item $([c]_{\sim},[d]_{\sim})\in\Cg^{\sm a/{\sim}}([a]_{\sim},[b]_{\sim})$.
\item $(c,d)\in(\Cg^{\sm a}(a,b)\jn{\sim})$.
\item There exist $k\in\omega$, unary polynomials $p_1,\dots,p_k\in \poln{1}{a}$ and $c=c_0,d_0,c_1,d_1,\dots,c_k,d_k=d$ in $A$ such that $\{d_{1-1},c_i\}=\{p_i(a),p_i(b)\}$ and $c_i\sim d_i$ for all $i\leq k$.
\end{enumerate}
\end{lm}

\begin{proof}
$(i)\Rightarrow(iii)$: If $([c]_{\sim},[d]_{\sim})\in(\Cg^{\sm a/{\sim}}([a]_{\sim},[b]_{\sim})$ then there exist $k\in\omega$, unary polynomials $q_1,\dots,q_k\in\poln{1}{a}/{\sim}$ and $[c]_\sim=[e_0]_{\sim},[e_1]_{\sim},\dots,[e_k]_{\sim}=[d]_{\sim}$ such that $\{[e_{i-1}]_{\sim},[e_i]_{\sim}\}=\{q_i([a]_{\sim}),q_i([b]_{\sim})\}$. We construct the polynomials $p_i\in \poln{1}{a}$ from $q_i$ by selecting representatives of $\sim$-classes so that whenever $q_i=t_i(x,[u_1]_{\sim},\dots,[u_n]_{\sim})$, we define $p_i=t_i(x,u_1,\dots,u_n)$. If $q_i([a]_{\sim})=[e_{i-1}]_\sim$, we say $d_{i-1}:=p_i(a)$ and $c_i:=p_i(b)$, while in the event that $q_i([a]_{\sim})=[e_i]_\sim$, we say $d_{i-1}:=p_i(b)$ and $c_i:=p_i(a)$. The obtained sequence of $c_i,d_j$ is as desired since from $q_i([u]_\sim)=q_{i+1}([v]_\sim)$ follows $p_i(u)\sim p_i(v)$.

$(iii)\Rightarrow(ii)$ is obvious, and $(ii)\Rightarrow(i)$ follows by taking the Mal'cev chains used in generating pairs in $\Cg^{\sm a}(a,b)$ and replacing the parameters with their $\sim$-classes.
\end{proof}

\begin{lm}\label{SMBCgVsimbelow}
Let $\vr v$ be a finitely generated variety of SMB algebras, $\m a\in \vr v$ an SMB algebra over $\sim$, $a,b,c,d\in A$ and $(c,d)\in(\Cg(a,b)\jn{\sim})$. Then there exist $e,f\in A$ such that $(c,e),(d,f)\in \Cg(a,b)$, $e\sim f$ and that $[c]_{\sim}\mt [d]_{\sim}\geq [e]_{\sim}$.
\end{lm}

\begin{proof}
Let $c=c_0,d_0,c_1,d_1,\dots,c_k,d_k=d$ satisfy $c_i\sim d_i$ and $(d_{i-1},c_i)\in \Cg(a,b)$ for all $i\leq n$. In view of Propositions~\ref{SMBtospec} and~\ref{RegSMBbase}, we can use as $\mt$ and $d$ the $\vr v$-terms such that $(A;\mt,d)$ is a regular SMB algebra over $\sim$. Define 
\[e_0:=c\text{, }e_\ell:=c_\ell\mt(c_{\ell-1}\mt(\dots\mt(c_1\mt c)\dots))\text{ and }e:=e_k.\]
Using Definition~\ref{regSMBdef} $(ii)$ and $d_{\ell-1}\sim c_{\ell-1}$, we conclude that 
\[
\begin{gathered}
d_{\ell-1}\mt e_{\ell-1}=d_{\ell-1}\mt (c_{\ell-1}\mt(c_{\ell-2}\mt(\dots\mt(c_1\mt c)\dots)))=\\
c_{\ell-1}\mt(c_{\ell-2}\mt(\dots\mt(c_1\mt c)\dots))=e_{\ell-1}.
\end{gathered}
\]
Hence we obtain
\[
\begin{gathered}
(e_\ell,e_{\ell-1})=(c_\ell\mt e_{\ell-1},d_{\ell-1}\mt e_{\ell-1})\in \Cg(a,b).
\end{gathered}
\]
By transitivity, $(c,e)=(e_0,e_k)\in \Cg(a,b)$.

A completely analogous proof shows that, if we define
\[f_k:=d\text{, }f_\ell:=d_\ell\mt(d_{\ell+1}\mt(\dots\mt(d_{k-1}\mt d)\dots))\text{ and }f:=f_0,\]
then $(f_{\ell},f_{\ell+1})\in\Cg(a,b)$ and, by transitivity, $(f,d)=(f_0,f_k)\in \Cg(a,b)$. It remains to note that 
\[
\begin{gathered}
{[e]_{\sim}}=[c_k]_{\sim}\mt[c_{k-1}]_{\sim}\mt\dots\mt[c_1]_{\sim}\mt[c_0]_{\sim}=\\
[d_k]_{\sim}\mt[d_{k-1}]_{\sim}\mt\dots\mt[d_1]_{\sim}\mt[d_0]_{\sim}=\\
[d_0]_{\sim}\mt[d_1]_{\sim}\mt\dots\mt[d_{k-1}]_{\sim}\mt[d_k]_{\sim}=[f]_{\sim},
\end{gathered}
\]
which completes the proof.
\end{proof}

\begin{cor}\label{SMBCgVsim}
Let $\vr v$ be a finitely generated variety of SMB algebras, $\m a\in \vr v$ an SMB algebra over $\sim$ and $a,b,c,d\in A$. Then $(c,d)\in \Cg(a,b)\jn{\sim}$ iff $(c,d\mt c),(d,c\mt d)\in \Cg(a,b)$, where $\mt$ is the $\vr v$-term such that $(A;\mt,d)$ is a regular SMB algebra over $\sim$. 
\end{cor}

\begin{proof}
$(\Leftarrow)$ follows from $(d\mt c)\sim(c\mt d)$. To prove $(\Rightarrow)$, let us assume $(c,d)\in \Cg(a,b)\jn{\sim}$. According to Lemma~\ref{SMBCgVsimbelow}, there exist $e',f'\in A$ such that $(c,e'),(d,f')\in\Cg(a,b)$, $e'\sim f'$ and $[e']_{\sim}\leq[c\mt d]_{\sim}$. Moreover, $(d\mt c)\mt c=d\mt c$, while $[d\mt c]_{\sim}\geq[e']_{\sim}$, so $(d\mt c)\mt e'=e'$ by Definition~\ref{regSMBdef}~$(ii)$. Hence,
\[(d\mt c,e')=((d\mt c)\mt c,(d\mt c)\mt e')\in \Cg(a,b),\]
and from $(c,e')\in\Cg(a,b)$, $(c,d\mt c)\in \Cg(a,b)$ follows by transitivity. The proof that $(d,c\mt d)\in\Cg(a,b)$ is analogous.
\end{proof}

\begin{cor}\label{SMBundersim}
Let $\vr v$ be a finitely generated variety of SMB algebras, $\m a\in \vr v$ an SMB algebra over $\sim$ and $a,b,c,d\in A$. Then 
$$\begin{gathered}
\Cg^{\sm a}(a,b)\cap \Cg^{\sm a}(c,d)\subseteq{\sim}\text{ iff}\\
\Cg^{\sm a/{\sim}}([a]_{\sim},[b]_{\sim})\cap \Cg^{\sm a/{\sim}}([c]_{\sim},[d]_{\sim})=0_{\sm a/{\sim}}.
\end{gathered}$$
\end{cor}

\begin{proof}
$(\Leftarrow)$ Let $(e,f)\in(\Cg^{\sm a}(a,b)\cap \Cg^{\sm a}(c,d))\setminus{\sim}$. Then we know $$([e]_{\sim},[f]_{\sim})\in\Cg^{\sm a/{\sim}}([a]_{\sim},[b]_{\sim})\text{ and }([e]_{\sim},[f]_{\sim})\in\Cg^{\sm a/{\sim}}([c]_{\sim},[d]_{\sim})$$ by using the same Mal'cev chains modulo the congruence $\sim$ as in $\Cg^{\sm a}$ (the terms from which the unary polynomials are constructed are the same, while each parameter $u\in A$ is replaced by $[u]_{\sim}$). Hence, \[\Cg^{\sm a/{\sim}}([a]_{\sim},[b]_{\sim})\cap \Cg^{\sm a/{\sim}}([c]_{\sim},[d]_{\sim})\neq 0_{\sm a/{\sim}}.\]

$(\Rightarrow)$ Assume that \[([e]_{\sim},[f]_{\sim})\in[\Cg^{\sm a/{\sim}}([a]_\sim,[b]_\sim)\cap \Cg^{\sm a/{\sim}}([c]_\sim,[d]_\sim)]\setminus (0_{\sm a/{\sim}}).\] 
By Lemma~\ref{SMBCgsim}, we obtain 
\[(e,f)\in\Cg^{\sm a}(a,b)\jn{\sim},\text{ }(e,f)\in\Cg^{\sm a}(c,d)\jn{\sim}\text{ and }e\nsim f.\] 
By Corollary~\ref{SMBCgVsim}, this implies 
\[(e,f\mt e),(f,e\mt f)\in \Cg^{\sm a}(a,b)\cap \Cg^{\sm a}(c,d)\text{ and }e\nsim f.\]
Since $e\mt f\sim f\mt e$, at least one of the pairs $(e,f\mt e)$ and $(f,e\mt f)$ is not in the same ${\sim}$-class, which finishes the proof.
\end{proof}

The following theorem summarizes what we can say in the direction of Park's Conjecture by the results we obtained thus far:

\begin{thm}
Let $\vr v$ be a finitely generated variety of SMB algebras and $\m a\in \vr v$ an SMB algebra over $\sim$.
\begin{enumerate}
\item For any $a,b,c,d\in A$,
\[
\begin{gathered}
{}[\Cg^{\sm a/{\sim}}([a]_{\sim},[b]_{\sim}),\Cg^{\sm a/{\sim}}([c]_{\sim},[d]_{\sim})]=0_{\sm a/{\sim}}\text{ iff }\\
[\Cg^{\sm a}(a,b),\Cg^{\sm a}(c,d)]\subseteq{\sim}.
\end{gathered}\]
\item There exists a first order formula $F(x,y,z,u)$ such that for all $\m a\in \vr v$ and all $a,b,c,d\in A$,
\[\m a\models F^{\sm a}(a,b,c,d)\text{ iff }[\Cg(a,b),\Cg(c,d)]\subseteq{\sim}.\]
\end{enumerate}
\end{thm}

\begin{proof}
$(i)$ First, note that $\m a/{\sim}$ is an algebra which has a semilattice term operation. Therefore, $\m a/{\sim}$ lies in a congruence meet-semidistributive variety. By \cite{kearnes-szendrei} and \cite{lipparini}, $[\alpha,\beta]=\alpha\cap\beta$ for congruences $\alpha,\beta\in\Cn{a}/{\sim}$. So, 
\[
\begin{gathered}
{}[\Cg^{\sm a/{\sim}}([a]_{\sim},[b]_{\sim}),\Cg^{\sm a/{\sim}}([c]_{\sim},[d]_{\sim})]=0_{\sm a/{\sim}}\text{ iff }\\
\Cg^{\sm a/{\sim}}([a]_{\sim},[b]_{\sim})\cap \Cg^{\sm a/{\sim}}([c]_{\sim},[d]_{\sim})=0_{\sm a/{\sim}},\text{ which implies}\\
\Cg^{\sm a}(a,b)\cap \Cg^{\sm a}(c,d)\subseteq{\sim}.
\end{gathered}
\]
The last implication was by Corollary~\ref{SMBundersim}. Of course, since the commutator is a subset of the intersection, we have proved that
\[
\begin{gathered}
{}[\Cg^{\sm a/{\sim}}([a]_{\sim},[b]_{\sim}),\Cg^{\sm a/{\sim}}([c]_{\sim},[d]_{\sim})]=0_{\sm a/{\sim}}\text{ implies}\\
[\Cg^{\sm a}(a,b), \Cg^{\sm a}(c,d)]\subseteq{\sim}.
\end{gathered}
\]

On the other hand, assume that $$[\Cg^{\sm a/{\sim}}([a]_{\sim},[b]_{\sim}),\Cg^{\sm a/{\sim}}([c]_{\sim},[d]_{\sim})]\neq 0_{\sm a/{\sim}}.$$
For shorter notation, denote $\alpha:=\Cg^{\sm a}(a,b)$ and $\beta:=\Cg^{\sm a}(c,d)$. By the above considerations and Corollary~\ref{SMBundersim}, there exist $(e,f)\in\alpha\cap\beta$ such that $e\nsim f$. Moreover, by considering $e\mt f$, which must be in the same $\alpha\cap\beta$-class as $e$ and $f$, we may assume, without loss of generality, that $[e]_{\sim}>[f]_{\sim}$. Consider the following $(\alpha,\beta)$-matrix:
\[
\left[\begin{array}{cc}
e\mt e & f\mt e\\
e\mt f & f\mt f
\end{array}\right]
=
\left[\begin{array}{cc}
e & f\mt e\\
f & f
\end{array}\right]
\]
(we used the term $t(x,y):=y\mt x$). The first row is not in $\sim$, since $[f]_{\sim}\mt [e]_{\sim} = [f]_{\sim}< [e]_{\sim}$, while the second row is in the equality relation, and thus in $[\alpha,\beta]$. The term condition $C(\alpha,\beta;[\alpha,\beta])$ implies that $(e,f\mt e)\in[\alpha,\beta]$, and therefore $[\alpha,\beta]$ is not a subset of ${\sim}$.

$(ii)$ Use the formula from Willard's proof that \[\Cg^{\sm a/{\sim}}([a]_{\sim},[b]_{\sim})\cap \Cg^{\sm a/{\sim}}([c]_{\sim},[d]_{\sim})=0_{\sm a/{\sim}}\]
is first-order definable. One can find it in \cite{Wreview}, Theorem 4.4 (see also Definition 2.1 of \cite{Wreview}).
\end{proof}

{\bf Remarks.} All results of this Section proved thus far, with a little work, can be extended to algebras in a finitely generated variety such that the generating finite algebra has a term reduct which is an SMB algebra.

At a first glance, one may think that we are close to first-order definition of $[\Cg(a,b),\Cg(c,d)]=0_{\sm a}$, which would be a major step toward proving Park's conjecture for SMB algebras. After all, we have found a first-order formula which forces $[\Cg(a,b),\Cg(c,d)]$ to be within the congruence $\sim$, and each $\sim$-block is a Mal'cev algebra. Also, the third author first-order defined $[\Cg(a,b),\Cg(c,d)]=0_{\sm a}$ in congruence modular, and therefore in Mal'cev, setting in \cite{Mmodular}. Unfortunately, there are further obstacles to overcome.

Recall that a difference term for $\vr v$ is a ternary term $t(x,y,z)$ such that $\vr v\models t(y,y,x)\approx x$ and moreover, for all $\m a\in\vr v$, all $\theta\in\Cn a$ and all $(x,y)\in\theta$, \[(t(x,y,y),x)\in [\theta,\theta].\]

Consider the following recent result:

\begin{thm}[Theorem 4.29 $(1)\Leftrightarrow(2)$ of \cite{Knew}]
A variety $\vr v$ has a difference term iff $\vr v$ has a Taylor term and for all $\m a\in \vr v$ and all $\alpha,\beta\in \Cn a$, $[\alpha,\beta]=[\beta,\alpha]$.
\end{thm}

Indeed, in the case of a Taylor algebra $\m a$ without a difference term, Kearnes in \cite{Knew} constructs a factor $\m b$ of such an algebra and two congruences $\alpha,\beta\in\Cn b$ such that $0_\sm b<\alpha<\beta$, $[\alpha,\beta]=0_\sm b$, but $[\beta,\alpha]=\alpha$ (cf. the proof of Theorem 4.28 of \cite{Knew}).

We will prove presently that there are SMB algebras without a difference term. The significance of this is, that in all previous proofs of Park's Conjecture, the formula $\varphi(x,y,z,u)$ which defined $[\Cg(a,b),\Cg(c,d)]=0_{\sm a}$ satisfied the sentence \[(\forall x,y,z,u)\varphi(x,y,z,u)\Leftrightarrow\varphi(z,u,x,y),\]
but in the case of SMB algebras, if we manage to find a formula $\varphi$ which defines $[\Cg(a,b),\Cg(c,d)]=0_{\sm a}$, we will probably have examples where this sentence fails. In order to prove that there exist SMB algebras without a difference term, we recall an old characterization of difference term varieties:

\begin{thm}[Theorem 1.1 of \cite{Kdiff}]\label{diffviatypes}
Let $\vr v$ be a variety such that $\m f_{\vr v}(2)$ is finite. $\vr v$ has a difference term iff for all finite $\m a\in\vr v$, $\mathbf{1}\notin\mathrm{typ}\{\m a\}$ and type $\mathbf{2}$ minimal sets of $\m a$ have empty tails.
\end{thm}

And now we give an example of an SMB variety without a difference term.

\begin{ex}
Let $A=\{0,1,2\}$ and let $d$ be a ternary operation on $A$ given by
\[d(x,y,z)=\left\{
\begin{array}{cc}
x+y+z,&\text{if }2\notin\{x,y,z\}\text{ and}\\
2,&\text{otherwise.}
\end{array}
\right.\] 
In the above formula, the addition is modulo 2. If we define $x\mt y:=d(x,x,y)$, it is not hard to verify that $\m a=(A;\mt,d)$ is a regular SMB algebra modulo $\sim$, where $x\sim y$ iff $x=y$ or $x,y\in \{0,1\}$. Moreover, $0_{\sm a}{\prec_2}\:{\sim}$ in $\Cn a$. Finally, we note that $A$ is an $(0_{\sm a},\sim)$-minimal algebra and its tail is $\{2\}$, i.e. it is nonempty. This follows since $2$ is an absorbing element for any term operation of $\m a$, so $2\in p(A)$, whenever $p$ is a nonconstant unary polynomial. Now Theorem~\ref{diffviatypes} implies that $\vr v(\m a)$ has no difference term.
\end{ex}

\end{document}